\documentclass[12pt, a4paper, leqno]{amsart}

\usepackage{amsmath}
\usepackage{amsfonts}
\usepackage{amssymb}
\usepackage{mathrsfs}
\usepackage{color}
\usepackage{inputenc}
\usepackage[T1]{fontenc} 
\usepackage{url}

\numberwithin{equation}{section}

\newtheorem{theorem}{Theorem}[section]
\newtheorem{definition}[theorem]{Definition}
\newtheorem{lemma}[theorem]{Lemma}

\newtheorem{corollary}[theorem]{Corollary}

\theoremstyle{remark}
\newtheorem{remark}[theorem]{Remark}

\theoremstyle{plain}

\newcommand{\N}{\mathbb{N}}

\newcommand{\Rn}{{\mathbb{R}^n}}

\newcommand{\Lh}{L^{p,\lambda}(\Rn)}

\newcommand{\Vz}{V_0L^{p,\lambda}(\Rn)}
\newcommand{\Vi}{V_\infty L^{p,\lambda}(\Rn)}
\newcommand{\Vast}{V^{(\ast)} L^{p,\lambda}(\Rn)}

\newcommand{\Vziast}{V_{0,\infty}^{(\ast)}L^{p,\lambda}(\Rn)}

\begin{document}


\title[Invariance of vanishing Morrey subspaces]{On the invariance of certain vanishing subspaces of Morrey spaces with respect to some classical operators}

\author[A. \c{C}. Alabalik]{Aysegul \c{C}. Alabalik}
\address{Department of Mathematics, Institute of Natural and Applied Sciences, Dicle University, 21280 Diyarbakir, Turkey}
\email{aysegulalabalik@gmail.com}

\author[A. Almeida]{Alexandre Almeida$^{*}$}
\address{Department of Mathematics, Center for Research and Development in Mathematics and Applications (CIDMA), University of Aveiro, 3810-193 Aveiro, Portugal}
\email{jaralmeida@ua.pt}

\author[S. Samko]{Stefan Samko}
\address{University of Algarve, Department of Mathematics, Campus de Gambelas, 8005-139 Faro, Portugal}
\email{ssamko@ualg.pt}

\thanks{$^*$ Corresponding author.}
\thanks{A. Almeida was partially supported by the Portuguese national funding agency for science, research and technology (FCT), within the Center for Research and Development in Mathematics and Applications (CIDMA), project UID/MAT/04106/2019. S. Samko was supported by Russian Foundation for Basic Research under the grants 19-01-00223 and 18-01-00094-a.}

\date{March 20, 2019}



\begin{abstract}
	We consider subspaces of Morrey spaces defined in terms of various vanishing properties of functions. Such subspaces were recently used to describe the closure of $C_0^\infty(\Rn)$ in Morrey norm. We show that these subspaces are invariant with respect to some classical operators of harmonic analysis, such as the Hardy-Littlewood maximal operator, singular type operators and Hardy operators. We also show that the vanishing properties defining those subspaces are preserved under the action of Riesz potential operators and fractional maximal operators.
\end{abstract} \maketitle

\section{\textbf{Introduction}}
Morrey spaces play an important role in the study of local behaviour and regularity properties of solutions to PDE, including heat equations and Navier-Stokes equations. We refer to \cite{Kato92,L-R16,Soft11,Tri13,Tri15} and references therein for further details.
For $1\leq p<\infty$, $0\leq \lambda \leq n$, the classical Morrey space $\Lh$ consists of all locally $p$-integrable functions $f$ on $\Rn$ with finite norm
\begin{equation}\label{homMorreynorm:general}
	\|f\|_{p,\lambda}:= \sup_{x\in\Rn,\, r>0} \mathfrak{M}_{p,\lambda}(f;x,r)^{1/p},
\end{equation}
where
\begin{equation}\label{def:modular}
	\mathfrak{M}_{p,\lambda}(f;x,r):= \frac{1}{r^\lambda} \int_{B(x,r)} |f(y)|^p\,dy\,, \ \ \ \ \ \ x\in\Rn, \ \ \ r>0.
\end{equation}
Straightforward calculations show that
$$
\|f(t\cdot)\|_{\Lh} = t^{\frac{\lambda-n}{p}}\,\|f\|_{\Lh}\,, \ \ \ \ t>0,
$$
which implies a modification of the scaling factor in comparison with $L^p$-spaces.

It is well known that the spaces $\Lh$ are non-separable if $\lambda>0$ (see \cite[Proposition~3.7]{RosTri15} for a proof). The lack of approximation tools for the entire Morrey space has motivated the introduction of appropriate subspaces, like the Zorko space (\cite{Zor86}) collecting all Morrey functions for which the translation is continuous in Morrey norm and vanishing spaces defined in terms of the vanishing properties $(V_0$), $(V_\infty$) and $(V^\ast$) defined below.

The theory of Morrey spaces goes back to Morrey \cite{Mor38} who considered related integral inequalities in the study of solutions to nonlinear elliptic equations. In the form of Banach spaces of functions, called thereafter Morrey spaces, the ideas of Morrey \cite{Mor38} were further developed by Campanato \cite{Cam64} and Peetre \cite{Pee69}. We refer to the books \cite{Adam15,Gia83,Pick13,Tay00,Tri13} and  the overview \cite{RafNSamSam13} for additional references and basic properties of these spaces and some of their generalizations. A discussion on Harmonic Analysis in Morrey spaces can be found in \cite{AdamXiao12}, \cite{RosTri14}, \cite{Tri15}.


Many classical operators from Harmonic Analysis such as maximal operators, singular operators, potential operators and Hardy operators, are known to be bounded in Morrey spaces. There are many papers in the literature dealing with this subject, including the case when the spaces and/or the operators have generalized parameters. We refer to the papers \cite{Adam75,AkbGulMus12,BurGogGulMus10,BurJaiTar11,ChiaFras87,Gul09,Kato92,KurNisSug00,LukMedPerNSam12,Miz91,Nakai94,PerNSam11,RosTri14,SawSugTan09,Shi06,Soft06,SugTan03}.


In this paper we are interested in studying the behavior of those classical operators in certain subspaces of Morrey spaces. We consider the following subspaces of $\Lh$. The class $\Vz$ consists of all those functions $f\in\Lh$ such that
\begin{equation}
	\lim_{r\to 0}\, \sup_{x\in\Rn}\, \mathfrak{M}_{p,\lambda}(f;x,r) =0. \tag{$V_0$}
\end{equation}
Similarly, $\Vi$ is the set of all $f\in\Lh$ such that
\begin{equation}
	\lim_{r\to \infty}\, \sup_{x\in\Rn}\, \mathfrak{M}_{p,\lambda}(f;x,r) =0. \tag{$V_\infty$}
\end{equation}
We also consider the set $\Vast$ consisting of all functions $f\in \Lh$ having the vanishing property
\begin{equation*}\label{vast-property}
	\lim_{N\to \infty} \mathcal{A}_{N,p}(f):= \lim_{N\to \infty} \sup_{x\in\Rn}\, \int_{B(x,1)} |f(y)|^p \, \chi_N(y)\,dy = 0,\tag{$V^\ast$}
\end{equation*}
where $\chi_N:= \chi_{\Rn \setminus B(0,N)}\,, \quad N\in\N.$

The three vanishing classes defined above are closed sets in $\Lh$ with respect to the norm \eqref{homMorreynorm:general}. The space $\Vz$, often called in the literature just by \emph{vanishing Morrey space}, was already introduced in \cite{ChiaFran92,Vit90,Vit93} in connection with applications to PDE. The subspaces $\Vi$ and $\Vast$ were recently introduced in \cite{AlmSam17,AlmSam18} to study the approximation problem by nice functions in Morrey spaces. Note that $\Vi$ was independently considered in \cite{YSY15} in the study of interpolation problems.

The subspace $\Vziast$, collecting those Morrey functions having all the vanishing properties $(V_0$), $(V_\infty$) and $(V^\ast$), provides an explicit description of the closure of $C_0^\infty(\Rn)$ in Morrey norm, see \cite[Theorem~5.3 and Corollary~5.4]{AlmSam17}.

The boundedness of classical operators in vanishing Morrey spaces at the origin was already studied in some papers, including the case of generalized parameters, see \cite{PerRagNSamWall12,Rag08,NSam13,NSam13p}. In particular, in \cite{NSam13} it was studied a class of sublinear singular type operators which includes the Hardy-Littlewood maximal function and Calder\'{o}n-Zygmund operators with standard kernels. The boundedness results in \cite{NSam13} were given in terms of Zygmund-type integral conditions on the function parameter defining the Morrey space.


Up to authors' knowledge, the boundedness of classical operators in the vanishing Morrey spaces $\Vi$ and $\Vast$ was not touched so far, apart some results in \cite[Theorem~3.8]{AlmSam17}, \cite[Corollary~4.3]{AlmSam18} where it was observed that convolution operators with integrable kernels are bounded in those subspaces. It is the main goal of this paper to show that the vanishing properties defining these subspaces are preserved under the action of many other operators from Harmonic Analysis, including maximal, singular, potential and Hardy operators. One of the key results is the invariance of the space $\Vast$ with respect to the Hardy-Littlewood maximal operator (cf. Theorem~\ref{the:VastMax}).

The paper is organized as follows. After some preliminaries on the operators under consideration, we give the main results in Section~\ref{sec:main}. The boundedness results on the spaces $\Vi$ and $\Vast$ are given in Sections~\ref{sec:Vi} and \ref{sec:Vast}, respectively. Section~\ref{sec:Vziast} is devoted to the study of the invariance of the smaller subspace $\Vziast$. Finally, we discuss additional results in Section~\ref{sec:hybrids} for some operators that can be seen as hybrids of potential and Hardy operators.

\section{Preliminaries}\label{sec:prelim}

We use the following notation: $B(x,r)$ is the open ball in
$\mathbb{R}^{n}$ centered at $x\in \mathbb{R}^{n}$ and radius $r>0$. The (Lebesgue) measure of a measurable set $E\subseteq {\mathbb{R}^{n}}$ is denoted by $|E|$ and $\chi_{E}$ denotes its characteristic function. The measure of the unit ball in $\Rn$ is simply denoted by $v_n$.
We use $c$ as a generic positive constant, i.e., a constant whose value may change with each appearance. The expression $A
\lesssim A$ means that $A\leq c\,B$ for some independent constant $c>0$, and $A\approx B$ means $A \lesssim B \lesssim A$.


As usual $C_0^\infty(\Rn)$ stands for the class of all complex-valued infinitely differentiable
functions on $\Rn$ with compact support, and $L^p(\Rn)$ denotes the classical Lebesgue space equipped with the usual norm.

\subsection{Some classical operators}

The following class of operators was introduced in \cite{NSam13}.

\begin{definition}
	Let $1<p<\infty$. A sublinear operator $T$ is called $p$-admissible singular type operator if it is bounded in $L^p(\Rn)$ and it
	satisfies a ``size condition'' of the form
	\begin{equation*}
		\chi_{B(x,r)}(z)\, \left| T\left(f\,\chi_{\Rn\setminus B(x,2r)}\right)(z)\right| \lesssim \chi_{B(x,r)}(z)\,\int_{\Rn\setminus B(x,2r)} \frac{|f(y)|}{|y-z|^n}dy
	\end{equation*}
	for every $x\in\Rn$ and $r>0$.
\end{definition}

An example of $p$-admissible singular type operators is the \emph{Hardy-Littlewood maximal operator}
$$
Mf(x):= \sup_{t>0} \frac{1}{|B(x,t)|} \int_{B(x,t)} |f(y)|dy, \quad x\in\Rn.
$$

It is well know that the maximal operator controls various other important operators of harmonic analysis. This is the case of the sharp maximal function
\begin{equation}\label{ineq:sharp-max}
	M^\sharp f(x):= \sup_{t>0} \frac{1}{|B(x,t)|} \int_{B(x,t)} |f(y)-f_{B(x,r)}|\,dy,
\end{equation}
with $f_B=\frac{1}{|B|} \int_B f(z)dz$. By straightforward calculations, we have
\begin{equation}\label{ineq:sharp-max}
	(M^\sharp f)(x) \leq 2 (Mf)(x), \quad x\in\Rn.
\end{equation}

The class above includes also singular integral operators $S$, defined by
\begin{equation}\label{def:singular}
	Sf(x):= \int_{\Rn} K(x,y)f(y)\,dy := \lim_{\varepsilon\to 0} \int_{|x-y|>\varepsilon} K(x,y)f(y)\,dy,
\end{equation}
which are bounded in $L^p(\Rn)$ and whose kernel satisfies
\begin{equation}\label{size-cond}
	\left|K(x,y)\right| \lesssim |x-y|^{-n}, \ \ \ \ \ \text{for all} \; x\neq y.
\end{equation}
This is the case of Calder\'{o}n-Zygmund operators with ``standard kernels'' (cf. \cite[p.~99]{Duo01}).

Other examples of $p$-admissible singular type operators are the multidimensional \emph{Hardy operators} $H$ and $\mathcal{H}$, defined by
$$
H f(x):= \frac{1}{|x|^{n}} \int_{|y|<|x|} f(y)\,dy \ \ \ \ \text{and} \ \ \ \ \mathcal{H} f(x):= \int_{|y|>|x|} \frac{f(y)}{|y|^{n}}\,dy.
$$
Using that $|x-y|<2|x|$ in the integral defining $H$, we get the pointwise estimate
\begin{equation}\label{ineq:HMaximal}
	H\big(|f|\big)(x) \leq 2^n v_n \, Mf(x), \quad x\in\Rn.
\end{equation}

We shall consider more general Hardy type operators, $H^\alpha$ and $\mathcal{H}^\alpha$, $0\leq \alpha <n$, defined for appropriate functions $f$ by

$$H^\alpha f(x):= |x|^{\alpha-n} \int_{|y|<|x|} f(y)\,dy \ \ \ \ \text{and} \ \ \ \ \mathcal{H}^\alpha f(x):= |x|^{\alpha} \int_{|y|>|x|} \frac{f(y)}{|y|^{n}}\,dy.
$$

It can be easily shown that the operator $H^\alpha$ is now dominated by the \emph{fractional maximal operator}
$$
M^\alpha f(x):= \sup_{t>0} \frac{1}{|B(x,t)|^{1-\frac{\alpha}{n}}} \int_{B(x,t)} |f(y)|dy, \quad x\in\Rn,
$$
which in turn can be estimated by the \emph{Riesz potential operator}
$$
I^\alpha f(x):=  \int_{\Rn} \frac{f(y)}{|x-y|^{n-\alpha}}\,dy, \quad x\in\Rn.
$$
More precisely, for $0<\alpha<n$ there holds
\begin{equation}\label{ineq:HMIalpha}
	\big|H^\alpha f(x)\big| \leq v_n\, 2^{n-\alpha}\, \big(M^\alpha f\big)(x)\leq 2^{n-\alpha} \,I^\alpha\big(|f|\big)(x), \quad\quad x\in\Rn.
\end{equation}

The Hardy operator $\mathcal{H}^\alpha$ is also dominated by the Riesz potential operator. The pointwise estimate below should be known, but since we did not find a reference in the literature we take the opportunity to give a simple proof of it.

\begin{lemma}
	If $0<\alpha<n$, then we have
	\begin{equation}\label{ineq:mathcalHIalpha}
		\big| \mathcal{H}^\alpha f(x)\big| \leq 2^{n-\alpha}\, I^{\alpha}\big(|f|\big)(x)\,, \ \ \ x\in\Rn.
	\end{equation}
\end{lemma}

\begin{proof}
	The inequality follows from the estimate
	$$ \frac{|x|^\alpha}{|y|^{n}} \leq \frac{2^{n-\alpha}}{|x-y|^{n-\alpha}}, \quad\quad \text{for}\quad |y|>|x|.$$
	Putting $t=\frac{|x|}{|y|}$ and $x^\prime=\frac{x}{|x|}$, $y^\prime=\frac{y}{|y|}$, the latter is equivalent to
	$$ t^\alpha \leq 2^{n-\alpha}\Big(\frac{|y|}{|x-y|}\Big)^{n-\alpha} \quad\quad \text{or} \quad\quad t^\alpha \leq \frac{2^{n-\alpha}}{|tx^\prime-y^\prime|^{n-\alpha}}$$
	which is a consequence of having $t<1$ and $|tx^\prime-y^\prime|\leq 2$.
\end{proof}

%

\subsection{Pointwise estimates for modulars}

The estimates given in the next two lemmas are taken from \cite[Theorems~4.1 and 4.3]{NSam13}.

\begin{lemma}\label{lem:modularA}
	Let $1<p<\infty$ and $0\leq\lambda<n$. If $T$ is a $p$-admissible sublinear singular type operator, then
	\begin{equation*}
		\mathfrak{M}_{p,\lambda}(T f;x,r) \lesssim r^{n-\lambda}\, \Big( \int_r^\infty t^{\frac{\lambda-n}{p}-1} \big(\mathfrak{M}_{p,\lambda}(f;x,t)\big)^{\frac{1}{p}} \,dt \Big)^{p}
	\end{equation*}
	with the implicit constant independent of $x\in\Rn$, $r>0$ and $f\in L^p_{loc}(\Rn)$.
\end{lemma}

Although the previous lemma is formulated for functions $f$ in $L^p_{loc}(\Rn)$, the finiteness of the right-hand side implies that the function $f$ must have already some prescribed behaviour at infinity which, together with the size condition, ensures the well-posedness of the operator $T$.

\begin{remark}
	The estimate given in Lemma~\ref{lem:modularA} implies that $p$-admissible sublinear singular type operators are bounded on $\Lh$ and also on the vanishing space $\Vz$, with $1<p<\infty$ and $0\leq \lambda<n$ (cf. \cite{NSam13}). In particular, all the operators $M$, $M^\sharp$, $S$, $H$ and $\mathcal{H}$ are bounded on both spaces $\Lh$ and $\Vz$.
\end{remark}


\begin{lemma}\label{lem:modularB}
	Let $0<\alpha<n$, $1<p<n/\alpha$, $1/q=1/p-\alpha/n$ and $0\leq\lambda,\mu<n$. Then
	\begin{equation*}
		\mathfrak{M}_{q,\mu}(I^\alpha f;x,r) \lesssim r^{n-\mu}\, \Big( \int_r^\infty t^{\frac{\lambda}{p}-\frac{n}{q}-1} \big(\mathfrak{M}_{p,\lambda}(f;x,t)\big)^{\frac{1}{p}} \,dt \Big)^{q}
	\end{equation*}
	with the implicit constant not depending on $x\in\Rn$, $r>0$ and $f\in L^p_{loc}(\Rn)$.
\end{lemma}

\begin{remark}
	The estimate given in Lemma~\ref{lem:modularB} implies a Sobolev-Spanne result on the $L^{p,\lambda} \to L^{q,\mu}$- boundedness of $I^\alpha$, under the additional assumptions $0\leq \lambda < n-\alpha p$ and $\lambda/p=\mu/q$. Moreover, the same estimate was used in \cite{NSam13} to show that the Riesz potential operator $I^\alpha$, and consequently the fractional maximal operator $M^\alpha$, are also bounded from the the vanishing space $\Vz$ into the vanishing space $V_0L^{q,\mu}(\Rn)$.
\end{remark}

\section{Main results}\label{sec:main}

In this section we show that both subspaces $\Vi$ and $\Vast$ are invariant with respect to the operators mentioned above. Since the the boundedness of such operators is already known in the whole Morrey space, we only have to show that the corresponding vanishing properties are preserved under the action of those operators.

\subsection{Preservation of the property $(V_\infty)$}\label{sec:Vi}


\begin{theorem}\label{the:Vinfty-singular}
	Let $1<p<\infty$ and $0\leq\lambda<n$. Then any $p$-admissible sublinear singular type operator $T$ is bounded in $\Vi$.
\end{theorem}

\begin{proof}
	The proof is based on the modular estimate given in Lemma~\ref{lem:modularA}. Let $f\in\Vi$. For any $\varepsilon>0$ there exists $R=R(\varepsilon)>0$ such that $\mathfrak{M}_{p,\lambda}(f;x,t)<\varepsilon$ for every $t\ge R$ and all $x\in\Rn$. Thus for $r\ge R$ we have
	\begin{equation*}
		\mathfrak{M}_{p,\lambda}(T f;x,r) \lesssim \varepsilon\,r^{n-\lambda}  \Big( \int_r^\infty t^{\frac{\lambda-n}{p}-1} dt \Big)^{p}\lesssim \varepsilon
	\end{equation*}
	with the implicit constants independent of $x$ and $r$. This shows that
	$$
	\lim_{r\to\infty}\sup_{x\in\Rn}\mathfrak{M}_{p,\lambda}(T f;x,r) = 0
	$$
	and hence $Tf\in\Vi$.
\end{proof}

\begin{corollary}\label{cor:Vi-maximal}
	If $1<p<\infty$ and $0\leq\lambda<n$, then the operators $M$, $M^\sharp$, $H$, $\mathcal{H}$ and $S$ are bounded in $\Vi$.
\end{corollary}

%
%
%

In the sequel $T^\alpha$, $0<\alpha<n$, stands for any of the operators $I^\alpha$, $M^\alpha$, $H^\alpha$ and $\mathcal{H}^\alpha$ above.

\begin{theorem}\label{the:Vinfty-Spanne}
	Let $0<\alpha<n$, $1<p<n/\alpha$, $0\leq\lambda<n-\alpha p$, $1/q=1/p-\alpha/n$ and $\lambda/p=\mu/q$. Then $T^\alpha$ is bounded from $\Vi$ into $V_\infty L^{q,\mu}(\Rn)$.
\end{theorem}

\begin{proof}
	Since $I^\alpha$ is bounded from $\Lh$ into $L^{q,\mu}(\Rn)$ (this is a well known result by Spanne and published in \cite{Pee69}), the same norm inequalities also hold for the operators $M^\alpha$, $H^\alpha$ and $\mathcal{H}^\alpha$ in virtue of the estimates \eqref{ineq:HMIalpha} and \eqref{ineq:mathcalHIalpha}. Moreover, by the same estimates, the preservation of the vanishing property at infinity by the maximal and the Hardy operators follows from the corresponding preservation by the action of the Riesz potential operator. Hence, it remains to show that
	$$ \sup_{x\in\Rn}\mathfrak{M}_{p,\lambda}(f;x,r) \to 0 \quad\quad \Rightarrow \quad\quad\sup_{x\in\Rn}\mathfrak{M}_{q,\mu}(I^\alpha f;x,r)\to 0 \quad \quad \text{as}\quad r\to\infty.$$
	This can be done as in the proof of Theorem~\ref{the:Vinfty-singular}, but now using the modular estimate given in Lemma~\ref{lem:modularB} instead of that in Lemma~\ref{lem:modularA}. Nevertheless, we write the proof in slightly different terms by applying the Lebesgue dominated convergence theorem. For any $x\in\Rn$ and $r>0$, we have
	\begin{eqnarray}\label{ineq:aux1}
		\mathfrak{M}_{q,\mu}(I^\alpha f;x,r) & \lesssim & r^{n-\mu}\, \Big( \int_r^\infty t^{\frac{\lambda}{p}-\frac{n}{q}-1} \big(\mathfrak{M}_{p,\lambda}(f;x,t)\big)^{\frac{1}{p}} \,dt \Big)^{q} \nonumber\\
		& = & \Big( \int_1^\infty s^{\frac{\lambda}{p}-\frac{n}{q}-1} \big(\mathfrak{M}_{p,\lambda}(f;x,rs)\big)^{\frac{1}{p}} \,ds \Big)^{q}.
	\end{eqnarray}
	Since the implicit constant does not depend on $x\in\Rn$ and $r>0$, and the integrand has admits an integrable dominant (note that $f\in\Lh$ and $\lambda/p<n/q$), then the right-hand side of \eqref{ineq:aux1} tends to $0$ as $r\to\infty$ (uniformly on $x$). This gives the desired result.
\end{proof}

\begin{theorem}\label{the:Vinfty-Adams}
	Let $0<\alpha<n$, $0\leq\lambda<n$, $1<p<(n-\lambda)/\alpha$ and $1/q=1/p-\alpha/(n-\lambda)$. Then $T^\alpha$ is bounded from $\Vi$ into $V_\infty L^{q,\lambda}(\Rn)$.
\end{theorem}

\begin{proof}
	As in the proof of Theorem~\ref{the:Vinfty-Spanne} we only have to prove the statement for the Riesz potential operator. The $L^{p,\lambda} \to L^{q,\lambda}$- boundedness of $I^\alpha$ is well known (cf. \cite{Adam75}). To show the preservation of the vanishing property at infinity we make use of the pointwise estimate of the Riesz potential operator in terms of the maximal function. It is known that there exists $c>0$ such that
	\begin{equation}\label{ineq:potential-maximal}
		\left| I^\alpha f(x)\right| \leq c\, \big(Mf(x)\big)^{\frac{p}{q}} \, \|f\|_{p,\lambda}^{1-\frac{p}{q}}
	\end{equation}
	for all $f\in\Lh$ and $x\in\Rn$ (see, for instance, \cite{BurtNSam16}, \cite[Chapter~7]{Adam15}). From \eqref{ineq:potential-maximal} we get
	$$\mathfrak{M}_{q,\lambda}\big(I^\alpha f;x,r\big) \lesssim \|f\|_{p,\lambda}^{q-p} \, \mathfrak{M}_{p,\lambda}\big(Mf;x,r\big)
	$$
	for $r>0$ and $x\in\Rn$. If $f\in\Vi$ then $Mf\in\Vi$ by Corollary~\ref{cor:Vi-maximal}. Consequently, we have $I^\alpha f \in V_\infty L^{q,\lambda}(\Rn)$ taking into account the previous estimate.
\end{proof}


\subsection{Preservation of the property $(V^\ast)$}\label{sec:Vast}

First we show that the space $\Vast$ is invariant under the action of the Hardy-Littlewood maximal operator.

\begin{theorem}\label{the:VastMax}
	Let $1<p<\infty$ and $0\leq\lambda<n$. Then the maximal operator $M$ is bounded in $\Vast$.
\end{theorem}

\begin{proof}
	Since $M$ is bounded in $\Lh$ (cf. \cite{ChiaFras87}) we only have to show that it preserves the vanishing property ($V^\ast$), that is
	\begin{equation*}
		\lim_{N\to \infty} \mathcal{A}_{N,p}(f) = 0 \quad \Rightarrow \quad \lim_{N\to \infty} \mathcal{A}_{N,p}(Mf) = 0.
	\end{equation*}
	Given $x\in\Rn$ and $N\in\N$, we split $f$ into
	$$
	f=f_1+f_2, \ \ \ \text{with} \ \ \ f_1:=f\,\chi_{\Omega_{x,N/2}}, \ \ \ f_2=f\,\chi_{\Rn\setminus\Omega_{x,N/2}},
	$$
	where, for short, we use the notation
	$$\Omega_{x,N}:= B(x,2) \cap \big(\Rn\setminus B(0,N)\big).$$
	
	Since $M$ is sublinear, we have
	\begin{equation}\label{vstar:split}
		\mathcal{A}_{N,p}(Mf) \lesssim \mathcal{A}_{N,p}(Mf_1) + \mathcal{A}_{N,p}(Mf_2).
	\end{equation}
	
	
	We show next that both quantities in the the right-hand side of \eqref{vstar:split} tend to zero as $N\to\infty$.
	The boundedness of $M$ in $L^p(\Rn)$ gives
	\begin{eqnarray}
		\int_{B(x,1)} \big(M(f_1)(y)\big)^p \chi_N(y)\,dy  & \leq & \int_{\Rn} \big(M(f_1)(y)\big)^p dy \nonumber\\
		& \lesssim & \int_{\Rn} |f_1(y)|^p dy \nonumber\\
		& = & \int_{\Omega_{x,N/2}} |f(y)|^p \,dy
	\end{eqnarray}
	with the implicit constant independent of $x$, $N$ and $f$. Since $f\in \Vast$, the right hand side above tends to zero uniformly on $x$ as $N\to\infty$ (note that the property $(V^\ast)$ does not depend on the particular value of the radius taken in the balls centered at $x$, cf. \cite[Lemma~3.4]{AlmSam17}). Therefore,
	$ \lim_{N\to\infty}\mathcal{A}_{N,p}(Mf_1) =0.$
	
	Now we deal with the second term in the sum in $\eqref{vstar:split}$. Let $\varepsilon>0$ be arbitrary. Then there exists $t_1>1$ such that $t^{\lambda-n} < \varepsilon$ for all $t\ge t_1$. For such fixed $t_1$, we have
	$$
	\int_{B(x,1)} \big(M(f_2)(y)\big)^p \chi_N(y)\,dy\lesssim I_1(x,N) + I_2(x,N)
	$$
	where
	$$
	I_1(x,N):= \int_{B(x,1)} \chi_N(y) \, \sup_{0<t< t_1} \left[\frac{1}{|B(y,t)|} \int_{B(y,t)} |f(z)|\,\chi_{\Rn\setminus\Omega_{x,N/2}}(z)\,dz\right]^p dy
	$$
	and
	$$
	I_2(x,N):= \int_{B(x,1)} \chi_N(y) \, \sup_{t\ge t_1} \left[\frac{1}{|B(y,t)|} \int_{B(y,t)} |f(z)|\,\chi_{\Rn\setminus\Omega_{x,N/2}}(z)\,dz\right]^p dy.
	$$
	First we estimate $I_2(x,N)$. By H\"{o}lder's inequality we have
	$$
	\frac{1}{|B(y,t)|} \int_{B(y,t)} |f(z)|\,dz \leq \frac{1}{|B(y,t)|^{1/p}} \, \|f\|_{L^p(B(y,t))}.
	$$
	Hence
	$$I_2(x,N) \leq \int_{B(x,1)} \, \sup_{t\ge t_1} \frac{1}{|B(y,t)|} \, \int_{B(y,t)} |f(z)|^p dz\,dy \lesssim \sup_{t\ge t_1} \frac{1}{t^{n-\lambda}} \, \|f\|^p_{p,\lambda} \leq \varepsilon \; \|f\|^p_{p,\lambda}.
	$$
	As regards $I_1(x,N)$ we have two different cases to be analysed for $z\in B(y,t)$ and $z\notin\Omega_{x,N/2}$. If $z\in B(0,N/2)$ then $t>|z-y|\geq |y|-|z|>N/2$. Thus there is no contribution to the supremum on $t\in(0,t_1)$ for $N\ge 2t_1$. If $z\notin B(x,2)$ then $t>|z-y| \ge |z-x| - |y-x| \geq 1$. Hence it remains to handle $I_1(x,N)$ when the supremum is taken over all $t\in(1,t_1)$. For such values of $t$ we have
	$$t^{-n} \int_{B(y,t)} |f(z)|\,\chi_{\Rn\setminus\Omega_{x,N/2}}(z)\,dz \leq  \int_{B(y,t_1)} |f(z)|\,dz = \int_{B(0,t_1)} |f(y-z)|\,dz.
	$$
	Using this, the Minkowski's inequality and a simple change of variables, we get
	\begin{eqnarray*}
		I_1(x,N) & \leq &  \int_{B(x,1)} \chi_N(y) \, \left[\int_{\Rn} \chi_{B(0,t_1)}(z)\,|f(y-z)|\,dz\right]^p dy\\
		& \leq & \left( \int_\Rn \chi_{B(0,t_1)}(z)  \left[ \sup_{v\in\Rn}\int_{B(v,1)} |f(u)|^p \,\chi_{N-|z|}(u)\,du \right]^{1/p}dz\right)^p\\
		& =: &  \left( \int_\Rn g_N(z)\,dz\right)^p
	\end{eqnarray*}
	with the interpretation
	$$\chi_a:=1\quad \text{if} \; a\leq 0 \quad\text{and} \quad \chi_a:=\chi_{\Rn\setminus B(0,a)}\; \text{if}\; a>0.$$
	This gives an uniform bound for $I_1(x,N)$.  Since $f\in \Lh$ and $g_N(z)$ has an integrable majorant (depending on $t_1$), an application of the Lebesgue convergence theorem shows that $\int_\Rn g_N(z)\,dz \rightarrow 0$  as $N\to\infty$, which implies that
	$$I_1(x,N) \rightarrow 0 \quad\;\text{uniformly on $x$,}\quad \text{as} \quad N\to\infty.$$
	The proof is complete.
\end{proof}

\begin{theorem}
	Let $1<p<\infty$ and $0\leq\lambda<n$. Then the operators $M^\sharp$, $H$ and $\mathcal{H}$ are bounded in $\Vast$.
\end{theorem}

\begin{proof}
	The boundedness of $M^\sharp$ and $H$ in $\Vast$ is a consequence of Theorem~\ref{the:VastMax} and inequalities \eqref{ineq:sharp-max}, \eqref{ineq:HMaximal}. The case of the Hardy operator $\mathcal{H}$ requires a different approach since it can not be estimated by the maximal function. Observing that
	$$\lim_{|z|\to\infty} \mathcal{H}f(z) = 0,$$
	then for any $\varepsilon>0$ there exists $N_\varepsilon\in\N$ such that $|\mathcal{H}f(y)| \leq (\varepsilon/{v_n})^{1/p}$ for all $|y|\ge N_\varepsilon$. Therefore
	$$
	\int_{\Rn} \big|\mathcal{H}f(y)\big|^p \chi_{B(x,1)\cap \{y:|y|>N\}}(y)\,dy \leq \varepsilon/{v_n} \int_{\Rn} \chi_{B(x,1)\cap \{y:|y|>N\}}(y)\,dy \leq \varepsilon
	$$
	for all  $N\geq N_\varepsilon$, uniformly on $x\in\Rn$. This shows that $\mathcal{A}_{N,p}\big(\mathcal{H}f\big) \to 0$ as $N\to\infty$.
\end{proof}

\begin{remark}
	As regards the preservation of the property $(V^\ast)$ by the action of the Hardy operator $\mathcal{H}$, one can formulate a stronger result. Indeed, as we can see from the proof above, $\mathcal{H}$ is bounded from the whole space Morrey $\Lh$ into $\Vast$.
\end{remark}

As in Section~\ref{sec:Vi} suppose again that $T^\alpha$, $0<\alpha<n$, denotes any of the operators $I^\alpha$, $M^\alpha$, $H^\alpha$ and $\mathcal{H}^\alpha$.

\begin{theorem}\label{the:Vast-Adams}
	Let $0<\alpha<n$, $0\leq\lambda<n$, $1<p<(n-\lambda)/\alpha$ and $1/q=1/p-\alpha/(n-\lambda)$. Then $T^\alpha$ is bounded from $\Vast$ into $V^{(\ast)} L^{q,\lambda}(\Rn)$.
\end{theorem}

\begin{proof}
	By \eqref{ineq:potential-maximal} we get
	$$\mathcal{A}_{N,q}\big(I^\alpha f\big) \lesssim \|f\|_{p,\lambda}^{q-p} \, \mathcal{A}_{N,p}(Mf)
	$$
	with the implicit constant independent of $f$ and $N\in\N$. If $f\in\Vast$ then $Mf\in\Vast$ by Theorem~\ref{the:VastMax}. Consequently, we also have $I^\alpha f \in V^{(\ast)} L^{q,\lambda}(\Rn)$ by the previous estimate.
\end{proof}

\subsection{Invariance of the closure of $C_0^\infty(\Rn)$}\label{sec:Vziast}

As shown in \cite{AlmSam17}, the subspace
$$\Vziast:= \Vz\cap\Vi\cap\Vast$$
coincides with the closure of the class $C_0^\infty(\Rn)$ in Morrey norm. Note that this closure plays an important role in harmonic analysis on Morrey spaces since its dual provides a predual space for Morrey spaces (cf. \cite{Adam15,AdamXiao12,RosTri15,Tri15}). Moreover (cf. \cite{AlmSam17}), we have the strict embeddings
$$
\Vziast \ \subsetneqq \ \Vz\cap\Vi \ \subsetneqq \ \Vz \ \subsetneqq \ \Lh.
$$

The next corollaries are immediate consequences of the results obtained in Sections~\ref{sec:Vi}, \ref{sec:Vast} and the already known corresponding boundedness in the vanishing space $\Vz$.

\begin{corollary}
	Let $1<p<\infty$ and $0\leq\lambda<n$. Then the operators $M$, $M^\sharp$, $H$ and $\mathcal{H}$ are bounded in $\Vziast$.
\end{corollary}

\begin{corollary}
	Let $0<\alpha<n$, $0\leq\lambda<n$, $1<p<(n-\lambda)/\alpha$ and $1/q=1/p-\alpha/(n-\lambda)$. Then $T^\alpha$ is bounded from $\Vziast$ into $V^{(\ast)}_{0,\infty} L^{q,\lambda}(\Rn)$.
\end{corollary}

We end this section with a further result for singular integral operators $S$ defined by \eqref{def:singular}, with the kernel satisfying \eqref{size-cond}. One knows that $S$ is bounded in $\Vz$, for $1<p<\infty$ and $0\leq \lambda < n$ (cf. \cite[Theorem~5.1]{NSam13}). On the other hand, we have seen that $S$ is also bounded in $\Vi$ (cf. Corollary~\ref{cor:Vi-maximal}). Unfortunately, we do not know whether $S$ preserves the vanishing property $(V^\ast)$. Nevertheless, we have the following result:

\begin{theorem}
	Let $1<p<\infty$ and $0\leq \lambda < n$. Then $S$ is bounded from $\Vziast$ into itself.
\end{theorem}

\begin{proof}
	For bounded compactly supported functions $f$, there holds
	\begin{equation}\label{ineq:aux2}
		|Sf(y)| \leq \frac{c}{|y|^n}
	\end{equation}
	for sufficiently large $|y|$, with $c>0$ not depending on $y$. In fact, suppose that $\text{supp}\, f \subset B(0,M)$. We have
	$$|y-z| \geq |y|-M \geq \frac{|y|}{2} \quad\quad \text{for} \;\; |z|\leq M \quad \text{and} \quad |y|\geq 2M.$$
	Hence \eqref{ineq:aux2} follows thanks to the size condition \eqref{size-cond}.
	
	Let now $f\in \Vziast$. Then there exists a sequence $(f_k) \subset C_0^\infty(\Rn)$ such that $f_k\to f$ in $\Lh$ as $k\to\infty$. By the continuity of $S$ in $\Lh$, we get
	$$ S f = S(\lim_{k\to\infty} f_k) = \lim_{k\to \infty} (S f_k).$$
	Thus we have $Sf\in \Vziast$ since this subspace is closed in $\Lh$.
\end{proof}

\section{Additional results}\label{sec:hybrids}

In this section we consider the following hybrids of Hardy and potential operators:
$$
K_\beta f(x):=\frac{1}{|x|^\beta} \int_{|y|<|x|} \frac{f(y)}{|x-y|^{n-\beta}}\,dy
$$
and
$$
\mathcal{K}_\beta f(x):=\int_{|y|>|x|} \frac{f(y)}{|y|^\beta |x-y|^{n-\beta}}\,dy
$$
(with $0<\beta\leq n$). As observed in \cite{PerNSamW16}, $K_\beta$ and $\mathcal{K}_\beta$ are integral operators bounded in $L^p(\Rn)$, $1<p<\infty$, with the corresponding kernels satisfying the size condition \eqref{size-cond}. Consequently, by Theorem~\ref{the:Vinfty-singular} we have the following result:
\begin{theorem}
	Let $1<p<\infty$, $0\leq\lambda<n$ and $0<\beta\leq n$. Then both operators $K_\beta$ and $\mathcal{K}_\beta$ are bounded on $\Vi$.
\end{theorem}

{\bf Acknowledgments.} A. Almeida was supported by CIDMA (Center for Research and Development in Mathematics and Applications) and FCT (Portuguese Foundation for Science and Technology) within project UID/MAT/04106/2019. The research of S. Samko was supported by Russian Foundation for Basic Research under the grants 19-01-00223 and 18-01-00094-a.

\bibliographystyle{amsplain}

\end{document}